\documentclass[11pt,reqno]{amsart}
\usepackage{graphicx}
\usepackage[utf8]{inputenc}
\usepackage{amsmath}
\usepackage{amsfonts}
\usepackage{amssymb}

\numberwithin{equation}{section}
\theoremstyle{plain}
\newtheorem{thm}{Theorem}[section]
\newtheorem{proposition}{Proposition}[section]
\newtheorem{corollary}{Corollary}[section]
\newtheorem*{lemma}{Lemma}

\theoremstyle{definition}
\newtheorem{definition}{Definition}
\theoremstyle{remark}
\newtheorem*{notation}{Notation}

\begin{document}
\title{On the mod-Gaussian convergence of a sum over primes}
\author{Martin Wahl}
\address{INSTITUT F\"{U}R MATHEMATIK, UNIVERSIT\"{A}T Z\"{U}RICH\\ WINTERTHURERSTRASSE 190, 8057 Z\"{U}RICH, SWITZERLAND}
\email{mwahl1983@yahoo.de}

\begin{abstract}
We prove mod-Gaussian convergence for a Dirichlet polynomial which approximates \(\operatorname{Im}\log\zeta(1/2+it)\). This Dirichlet polynomial is sufficiently long to deduce Selberg's central limit theorem with an explicit error term. Moreover, assuming the Riemann hypothesis, we apply the theory of the Riemann zeta-function to extend this mod-Gaussian convergence to the complex plane. From this we obtain that $\operatorname{Im}\log\zeta(1/2+it)$ satisfies a large deviation principle on the critical line. Results about the moments of the Riemann zeta-function follow. 
\end{abstract}
\maketitle
\section{Introduction} 
In this paper we study the distribution of values taken by $\log\zeta(1/2+it)$. A breakthrough was achieved by Selberg who showed that as \(t\) varies in \([T,2T]\), the distribution of \((\operatorname{Re}\log\zeta(1/2+it),\operatorname{Im}\log\zeta(1/2+it))\) is approximately Gaussian, with independent components each having expectation \(0\) and variance \((\log\log T)/2\). More precisely, he proved a central limit theorem which,  by the L\'{e}vy continuity theorem, is equivalent to the statement that
\begin{equation}\label{eq:i1}
  \frac{1}{T}\int_T^{2T} e^{iu\frac{\operatorname{Re}\log\zeta(1/2+it)}
  {\sqrt{(\log\log  T)/2}}
  +iv\frac{\operatorname{Im}\log\zeta(1/2+it)}{\sqrt{(\log\log T)/2}}}dt
  \rightarrow e^{-u^2/2-v^2/2},
\end{equation}
as $T\rightarrow\infty$, for all real numbers $u$ and $v$. For the case of $\operatorname{Im}\log\zeta(1/2+it)$ see \cite{S}, \cite{S2}, and also the work of Ghosh \cite{GO2}. The general case is investigated for instance in the book of Joyner \cite{J}. Some of Selberg's more recent results, for example about the rate of convergence, can be found in \cite{S3} and the thesis of Tsang \cite{T}. Initially, Selberg obtained the asymptotics of the joint moments which
lead to~\eqref{eq:i1} by the method of moments. A more effective approach, applied in our analysis, too, is treated in the work of Bombieri and Hejhal \cite{BH}. A central limit theorem for the sum over primes \((1/\sqrt{(\log\log x)/2})\sum_{p\leq x}p^{-1/2-iU_T}\), $U_T$ being random variables uniformly distributed on $[T,2T]$, \(\log x=\log T/(\log\log T)^{1/4}\), follows from the mean value theorem of Montgomery and Vaughan and the method of moments. To complete the proof (see \cite[Lemma 3 and Corallary]{BH}), they showed that the $L^1$-norm of $\log\zeta(1/2+iU_T)-\sum_{p\leq x}p^{-1/2-iU_T}$ is sufficiently small. 

The convergence in~\eqref{eq:i1} is also a consequence of a conjecture on the behaviour of the moments of the Riemann zeta-function on the critical line (see, e.g., the work of Keating and Snaith \cite{KS} and the references therein). It asserts that
\begin{multline}\label{eq:ksmc}
  e^{(z_1^2+z_2^2)(\log\log T)/4}\frac{1}{T}\int_T^{2T} e^{iz_1\operatorname{Re}\log\zeta(1/2+it)
 +iz_2\operatorname{Im}\log\zeta(1/2+it)}dt\\\rightarrow \Phi_g(z_1,z_2)\Phi_a(z_1,z_2)\ \ \ as\ T\rightarrow \infty
\end{multline}
locally uniformly for \(z_1,z_2\in\mathbb{C}\) with \(\operatorname{Re}(iz_1)>-1\) and analytic functions $\Phi_g$, $\Phi_a$ (see \cite[Conjecture 9]{KN} and also \cite[Conjecture 1]{GHK}). This type of convergence was introduced in \cite{JKN} where it is called mod-Gaussian convergence. 

A precise form of the function \(\Phi_g\) was conjectured by Keating and Snaith and is based on calculations in the theory of random matrices (see \cite{KS}, \cite[formula (18)]{KN}). The arithmetic factor \(\Phi_a\) can be explained, e.g., by computing the characteristic function of $\sum_{n\leq x}\Lambda(n)/(n^{1/2+iU_T}\log n)$ (see \cite[Theorem 2]{GHK}, where $x$ has to be $O((\log T)^{2-\epsilon})$) or of the corresponding stochastic model (replace \(\left\{p^{iU_T}\right\}_{p\in\mathbb{P}}\) by an independent sequence of random variables uniformly distributed on the unit circle, see \cite[Example 4]{KN}). %Note that $\Lambda(n)$ denotes as usual the von Mangoldt function.
 
In this paper we further investigate the distribution of the sum over primes \(\sum_{p\leq x}p^{-1/2-it}\) as \(t\) varies in \([T,2T]\) and its consequences on the distribution of values of the Riemann zeta-function on the critical line. Here, we will restrict ourselves to the case of $\operatorname{Im}\log\zeta(1/2+it)$. Note that some of the arguments cannot be applied to the case of $\operatorname{Re}\log\zeta(1/2+it)$.
It is our first aim to establish mod-Gaussian convergence if \(x\) fulfills certain conditions. Precisely, in Section~\ref{mcsp} we prove the following:

\begin{thm}\label{thma} Let $x=e^{\log T/N}$ and $N$ such that $x\rightarrow\infty$ and $N/\log\log T\rightarrow\infty$ as $T\rightarrow\infty$. Then
\begin{equation}\label{eq:i3}
  e^{u^2(\log\log x+\gamma)/4}\frac{1}{T}\int_{T}^{2T} e^{iu\sum_{p\leq x}
  \frac{\sin (t\log p)}{\sqrt{p}}}dt\rightarrow \Phi(u)\ \ \ as\ T\rightarrow\infty
\end{equation}
locally uniformly for $u\in\mathbb{R}$. Here, \(\gamma\) denotes Euler's constant and \(\Phi\) is the analytic function given by
\begin{equation}\label{eq:lf}
  \Phi(u)=\prod_{p\in\mathbb{P}}\Big(1-\frac{1}{p}\Big)^{-u^2/4}
  J_0\Big(\frac{u}{\sqrt{p}}\Big),
\end{equation}
where \(J_0\) denotes the zeroth Bessel function (see, e.g., Section~\ref{bessel}).
\end{thm}

One interesting point of the result seems to be the size of $x$. It can be chosen large enough to obtain Selberg's central limit theorem with Selberg's explicit error term (see \cite[Theorem 2]{S3} and Appendix A). Moreover, we obtain the following improvement of~\eqref{eq:i1}: 

\begin{corollary}\label{cor1} Assume RH. For $T$ sufficiently large, we have
\begin{equation*}
  \frac{1}{T}\int_T^{2T} e^{iv\frac{\operatorname{Im}\log\zeta(1/2+it)}{\sqrt{(\log\log T)/2}}}dt=e^{-v^2/2}
  +v^2O\Big(\frac{\log\log\log T}{\log\log T}\Big)+O(1/\log T)
\end{equation*}
uniformly for \(|v|\leq\sqrt{\log\log T/\log\log\log T}\).
\end{corollary}

In Section~\ref{mccp} we deal with the question if the convergence in Theorem~\ref{thma} can be extended to the complex plane. Assuming the Riemann hypothesis, we prove such a result for a weighted sum over primes.  

\begin{thm}\label{thmb} Assume RH. Let $x=e^{\log T/N}$ and $N$ such that $x\rightarrow\infty$ and $N/\log\log T\rightarrow\infty$ as $T\rightarrow\infty$. Furthermore, let $f$ be the function $f(u)=(\pi u/2)\cot(\pi u/2)$ and $\gamma_f=-0.1080\dots$ be the constant defined by $\prod_{p\leq x}(1-f^2(\log p/\log x)/p)=(e^{-\gamma_f}/\log x)(1+o(1))$. Then
\begin{equation*}
  e^{z^2(\log\log x+\gamma_f)/4}\frac{1}{T}\int_{T}^{2T} e^{iz\sum_{p\leq x}\frac{\sin (t\log p)}{\sqrt{p}}f\left(\frac{\log p}{\log x}\right)}dt
  \rightarrow \Phi(z)\ \ \ as\ T\rightarrow\infty
\end{equation*}
locally uniformly for $z\in\mathbb{C}$, where \(\Phi\) is given by~\eqref{eq:lf}.
\end{thm}
More general sums are possible as well (see \cite[Lemma 1]{G} and \cite[Lemma 1]{BH}). For the evaluation of $\gamma_f$ see \cite[proof of Lemma 6]{G}.

The crucial step from Theorem~\ref{thma} to Theorem~\ref{thmb} is an estimate of the exponential moments of the above sum. For this purpose let \(x\leq T^2\) and \(h\in\mathbb{R}\). Assuming the Riemann hypothesis, we then show that there exist constants \(C,C'\), and \(C''\) such that 
\begin{equation*}
 \frac{1}{T}\int_T^{2T} e^{h\sum_{n\leq x}\frac{\Lambda(n)}{\log n}\frac{\sin(t\log n)}{\sqrt{n}}f
 \left(\frac{\log n}{\log x}\right)}dt\leq C''e^{C|h|\frac{\log T}{\log x}+C'h^2\log\log T}.
\end{equation*}
Note that this inequality, which is almost a subgaussian bound, is valid beyond the range which is contained in Theorem~\ref{thma} and Theorem~\ref{thmb}.

We turn to the applications of Theorem~\ref{thmb}. 
As described above, Theorem~\ref{thma} can be used to obtain results in connection with the central limit theorem. 
In addition, Theorem~\ref{thmb} yields large deviations results. Applying the G\"{a}rtner-Ellis theorem and Theorem~\ref{thmb}, one obtains a large deviation principle (see \cite[chapter 1.2]{DZ} or Appendix~\ref{ldt} for the definition of the large deviation principle) from which we will deduce the following two Corollaries.

\begin{corollary} \label{cor2} Assume RH. Let $U_T$ be random variables uniformly distributed on $[T,2T]$. Then the family $(1/((\log\log T)/2))\operatorname{Im}\log\zeta(1/2+iU_T)$  satisfies the large deviation principle with the speed $1/((\log\log T)/2)$ and the rate function $I(h)=h^2/2$.
For instance,
\begin{multline}
\frac{1}{(\log\log T)/2}\log\Big(\frac{1}{T}\lambda(\{t\in[T,2T]:\operatorname{Im}\log\zeta(1/2+it)\geq h(\log\log T)/2\})\Big)\\\rightarrow-h^2/2\ \ \ as\ T\rightarrow\infty,\label{eq:ldpf}
\end{multline}
where $h>0$ and $\lambda$ denotes the Lebesgue measure.
\end{corollary}
\begin{corollary} \label{cor3} Assume RH. Let $h\in\mathbb{R}$. Then
\begin{equation*}
 \frac{1}{(\log\log T)/2}\log\bigg(\frac{1}{T}\int_T^{2T}e^{h\operatorname{Im}\log\zeta(1/2+it)}dt\bigg)\rightarrow h^2/2\ \ \ as\ T\rightarrow\infty.
\end{equation*}
\end{corollary}
Related papers which also discuss large deviations results are the work of Radziwi\l\l\ \cite{Rad}, who extended the range of Selberg's central limit theorem for $\operatorname{Re}\log\zeta(1/2+it)$ and the work of Soundararajan \cite{SO}, who proved large deviation bounds for $\operatorname{Re}\log\zeta(1/2+it)$. In fact, Soundararajan \cite[Corollary A]{SO} completed the proof of Corollary~\ref{cor3} in the case of $\operatorname{Re}\log\zeta(1/2+it)$ by proving the upper bound. The result can be stated as follows. For all $\epsilon>0$ and all $h>0$ we have
$(\log T)^{h^2-\epsilon}\ll_{h,\epsilon}\int_T^{2T}|\zeta(1/2+it)|^{2h}dt\ll_{h,\epsilon}(\log T)^{h^2+\epsilon}$. Note that the proof of the upper bound also applies to the case of $\operatorname{Im}\log\zeta(1/2+it)$ and that we apply a slightly weaker upper bound in the proofs of Theorem~\ref{thmb} and Corollary~\ref{cor3}. 
\begin{notation} For $y\geq 2$ and a function $g:[0,1]\rightarrow [0,1]$, we define
\begin{align*}
 &\Sigma_{g,y}(t)=\sum_{p\leq y}\frac{1}{p^{1/2+it}}g\Big(\frac{\log p}{\log y}\Big),\\
 &\Sigma_{g,y}^*(t)=\sum_{n\leq y}\frac{\Lambda(n)}{\log n}\frac{1}{n^{1/2+it}}g\Big(\frac{\log n}{\log y}\Big),
\end{align*}
$r_{g,y}(t)=\log\zeta(1/2+it)-\Sigma_{g,y}(t)$, and $r_{g,y}^*(t)=\log\zeta(1/2+it)-\Sigma_{g,y}^*(t)$.
\end{notation}
\section{Moments of a sum over primes}\label{moments}
Section~\ref{moments} is devoted to some standard mean value calculations. In doing so, we will apply the following generalization of the mean value theorem of Montgomery and Vaughan contained in \cite[Theorem 1.4.3]{R} (see also \cite[Lemma 3.1]{T}). Let $a_1,\dots,a_M$ and $b_1,\dots,b_M$ be complex numbers, $M\geq 2$, and let $T>0$. Then
\begin{multline}\label{eq:ram}
 \frac{1}{T}\int_T^{2T}\bigg(\sum_{m\leq M}{a_m m^{-it}}\bigg)
 \bigg(\overline{\sum_{m\leq M}{b_m m^{-it}}}\bigg)dt\\= 
 \sum_{m\leq M}a_m\overline{b}_m+\theta\frac{2D}{T}\sqrt{\sum_{m\leq M}m|a_m|^2}
 \sqrt{\sum_{m\leq M}m|b_m|^2},
\end{multline}
where $\theta$ depends on the various parameters but satisfies $|\theta|\leq 1$ and $D$ is the universal constant in \cite[Theorem 1.4.3]{R}.
\begin{proposition}\label{prop1} Let $x\geq 2$ and $T>0$ be real numbers, $k$ be a nonnegative integer, and $p_1,\dots,p_n$ be the prime numbers not exceeding $x$. Then
\begin{multline}\label{eq:im}
  \frac{1}{T}\int_T^{2T}\bigg(\sum_{p\leq x}\frac{\sin (t\log p)}{\sqrt{p}}\bigg)^{2k}dt\\
  =\frac{1}{2^{2k}}\binom{2k}{k}\sum_{\lambda_1+\cdots+\lambda_n=k}
  \bigg(\frac{k!}{\lambda_1!\cdots\lambda_n!}\bigg)^2p_1^{-\lambda_1}\cdots p_n^{-\lambda_n} 
  +\theta\frac{2D}{T}\sqrt{n^{2k}(2k)!}
\end{multline}
and $|(1/T)\int_T^{2T}(\sum_{p\leq x}\sin(t\log p)/\sqrt{p})^{2k+1}dt|\leq(2D/T)\sqrt{n^{2k+1}(2k+1)!}$ with $|\theta|\leq 1$ and $D$ the constant in~\eqref{eq:ram}. Furthermore, the main term in~\eqref{eq:im} is bounded by $((2k)!/2^{2k}k!)(\sum_{p\leq x}1/p)^k$.
\end{proposition}

\begin{proof}
From $\sin (t\log p)=(p^{it}-p^{-it})/2i$, we obtain
\begin{multline}\label{eq:binth}
  \frac{1}{T}\int_T^{2T}\bigg(\sum_{p\leq x}\frac{\sin (t\log p)} 
  {\sqrt{p}}\bigg)^{k}dt\\
  =\frac{1}{(2i)^k}\sum_{j=0}^k\binom{k}{j}
  \frac{(-1)^j}{T}\int_T^{2T}\bigg(\sum_{p\leq x}{\frac{1}{p^{1/2+it}}}\bigg)^j
 \bigg(\overline{\sum_{p\leq x}\frac{1}{p^{1/2+it}}}\bigg)^{k-j}dt.
\end{multline}
For $j=1,\dots,k$ the multinomial theorem yields
\begin{equation}\label{eq:multth}
 \bigg(\sum_{p\leq x}{\frac{1}{p^{1/2+it}}}\bigg)^j
 =\sum_{\lambda_1+\cdots+\lambda_n=j}\frac{j!}{\lambda_1!\cdots\lambda_n!}
 (p_1^{-\lambda_1}\cdots p_n^{-\lambda_n})^{1/2+it}.
\end{equation}
If we plug in~\eqref{eq:multth} into~\eqref{eq:binth} with $k$ replaced by $2k$, we obtain from~\eqref{eq:ram} that
\begin{align}
  \frac{1}{T}\int_T^{2T}\bigg(\sum_{p\leq x}&\frac{\sin (t\log p)}{\sqrt{p}}\bigg)^{2k}dt\label{eq:im2}\\
  =\frac{1}{2^{2k}}&\binom{2k}{k}\sum_{\lambda_1+\cdots+\lambda_n=k}
  \bigg(\frac{k!}{\lambda_1!\cdots\lambda_n!}\bigg)^2
  p_1^{-\lambda_1}\cdots p_n^{-\lambda_n}\nonumber \\
  +\frac{\theta2D}{2^{2k}T}\sum_{j=0}^{2k}\binom{2k}{j}&
  \sqrt{\sum_{\lambda_1+\cdots+\lambda_n=j}\bigg(\frac{j!}{\lambda_1!\cdots\lambda_n!}\bigg)^2}
  \sqrt{\sum_{\lambda_1+\cdots+\lambda_n=2k-j}\bigg(\frac{(2k-j)!}{\lambda_1!\cdots\lambda_n!}\bigg)^2}\nonumber
\end{align}
with \(|\theta|\leq 1\). Applying $j!/(\lambda_1!\cdots\lambda_n!)\leq j!$, $j=0,\dots,2k$, we bound the absolute value of the remainder by
\begin{equation}\label{eq:lll}
 \frac{2D}{2^{2k}T}\sum_{j=0}^{2k}\binom{2k}{j}\sqrt{n^jj!\,n^{2k-j}(2k-j)!}\leq
 \frac{2D}{T}\sqrt{n^{2k}(2k)!}.
\end{equation}
The main term in~\eqref{eq:im} can be bounded similarly. As in~\eqref{eq:im2} and~\eqref{eq:lll}, we also bound the $(2k+1)$th moment. Note that there is no main term in this case. This completes the proof.
\end{proof} 
We want to compare these mean value estimates to some random variables expectations.
Therefore, let \(X_1,X_2,\dots\) be an i.i.d. sequence of random variables uniformly distributed on the unit circle and let $p_1,\dots,p_n$ be the primes not exceeding $x$. Then
\begin{equation}\label{eq:ism}
 \mathbb{E}\bigg[\bigg(\sum_{i=1}^{n}\frac{\operatorname{Im}X_i}{\sqrt{p_i}}\bigg)^{2k}\bigg] = \frac{1}{2^{2k}}\binom{2k}{k}\sum_{\lambda_1+\cdots+\lambda_n=k}
 \bigg(\frac{k!}{\lambda_1!\cdots\lambda_n!}\bigg)^2p_1^{-\lambda_1}\cdots p_n^{-\lambda_n}
\end{equation}
and $\mathbb{E}\left[(\sum_{i=1}^n\operatorname{Im}X_i/\sqrt{p_i})^{2k+1}\right]=0$. To prove this, we replace $\sin(t\log p)$ by $\operatorname{Im}X_i$ and integration by expectation in~\eqref{eq:binth} and~\eqref{eq:multth} and then apply the formula $\mathbb{E}[X_1^{\lambda_1}\cdots X_n^{\lambda_n}X_1^{-\mu_1}\cdots
X_n^{-\mu_n}]=1$ if $\lambda_j=\mu_j$ for all $j=1,\dots,n$ and $=0$ else.

\section{Bessel functions}\label{bessel}

The Bessel functions appear in the Fourier expansion of the function \(e^{iz\sin \theta}\),
\begin{equation}\label{eq:bf}
  e^{iz\sin \theta}=\sum_{k=-\infty}^{\infty}J_k(z)e^{ik\theta}.
\end{equation}
Explicitly the $k$th Bessel function \(J_k(z)\) is given by
\begin{equation}\label{eq:infs}
  J_k(z)=\sum_{n=0}^\infty\frac{(-1)^n(z/2)^{k+2n}}{n!(k+n)!}
\end{equation}
for \(k\geq 0\) and given by the relation \(J_{k}(z)=(-1)^kJ_{-k}(z)\) for \(k<0\) (for these and more facts about Bessel functions see, e.g., the book of Andrews, Askey, and Roy \cite{A}). This section is devoted to the following mod-Gaussian convergence result (compare to \cite[Proposition 4.1]{JKN}).
\begin{proposition}\label{prop2} Let $X_1,X_2,\dots$ be an i.i.d. sequence of random variables uniformly distributed on the unit circle and let $p_1,p_2,\dots$ be the increasing sequence of all primes. Then
\begin{equation}\label{eq:cmc}
  e^{z^2(\log\log x+\gamma)/4}\mathbb{E} \Big[e^{iz\sum_{j=1}^{\pi(x)}\frac{\operatorname{Im}X_j}{\sqrt{p_j}}}\Big]\rightarrow 
  \Phi(z)\ \ \ \ \ \ as\ x\rightarrow\infty
\end{equation}
locally uniformly for $z\in\mathbb{C}$. Here, $\gamma$ denotes Euler's constant, $\pi(x)$ denotes the number of primes not exceeding $x$, and $\Phi(z)$ is  given by~\eqref{eq:lf}.
\end{proposition}
\begin{proof}
By~\eqref{eq:bf}, we have
\begin{equation}\label{eq:esm}
  \mathbb{E}\big[e^{iz\operatorname{Im} X_1}\big]=\frac{1}{2\pi}\int_0^{2\pi} e^{iz\sin \theta}d\theta=J_0(z).
\end{equation}
Applying the independence of the $X_j$'s,~\eqref{eq:esm}, and finally Merten's formula $\prod_{p\leq x}(1-1/p)=(e^{-\gamma}/\log x)(1+o(1))$, we obtain that the left hand side of~\eqref{eq:cmc} is equal to 
\begin{equation*}
 e^{z^2(\log\log x+\gamma)/4}\prod_{p\leq x}J_0\Big(\frac{z}{\sqrt{p}}\Big)=(1+o(1))^{z^2/4}\prod_{p\leq x}\Big(1-\frac{1}{p}\Big)^{-z^2/4}
  J_0\Big(\frac{z}{\sqrt{p}}\Big).
\end{equation*}
It remains to show that the above product converges to $\Phi(z)$, locally uniformly for $z\in\mathbb{C}$. This follows from the fact that the product $\Phi(z)$ is normally convergent (see \cite[Chapter IV.1, especially Remark IV.1.7]{FB}). This completes the proof.
\end{proof}
Consider the random variables $\operatorname{Im}\Sigma_{1,x}(-U_T)$, $U_T$ being random variables uniformly distributed on $[T,2T]$. As mentioned in the Introduction, one can use the method of moments to deduce that, as $x\rightarrow\infty$, $x=T^{o(1)}$, $(1/\sqrt{(\log\log x)/2})\operatorname{Im}\Sigma_{1,x}(-U_T)$ converges in distribution to a Gaussian random variable with expectation $0$ and variance $1$ (see \cite[Proof of Theorem B]{BH}). We will generalize this result by considering the cumulants of $\operatorname{Im}\Sigma_{1,x}(-U_T)$. 

If $Y$ is a real random variable such that $\mathbb{E}[e^{zY}]$ exists and is finite for all $z\in\mathbb{C}$, $\mathbb{E}[e^{zY}]$ is an analytic function and there exists a neighbourhood of $0$ where $\log \mathbb{E}[e^{zY}]=\sum_{m=1}^{\infty}\kappa_m(Y)z^m/m!$. The coefficients $\kappa_m(Y)$, $m\geq 1$, are called the cumulants of $Y$. Thus, $\kappa_m(Y)$ is equal to the $m$th derivative of $\log\mathbb{E}[e^{zY}]$ evaluated  at $0$. 
\begin{corollary}\label{cum} Let $x=e^{\log T/N}$ and $N$ such that $x\rightarrow\infty$ and $N\rightarrow\infty$ as $T\rightarrow\infty$ and let $U_T$ be random variables uniformly distributed on $[T,2T]$. Then, as $T\rightarrow\infty$,
$\kappa_2(\operatorname{Im}\Sigma_{1,x}(-U_T))-(\log\log x+\gamma)/2\rightarrow c_2$ and for $m\neq 2$ $\kappa_m(\operatorname{Im}\Sigma_{1,x}(-U_T))\rightarrow c_m$, where the $c_m$'s are defined by the series expansion
$\log\Phi(-iz)=\sum_{m=1}^{\infty}c_mz^m/m!$, for $z$ in a neighbourhood of $0$.
\end{corollary}
\begin{proof}
By the construction of $\Phi$, there exists a real number $0<r\leq 1$ such that for $|z|\leq r$, $\log\Phi(z)=\sum_{p\in\mathbb{P}}((-z^2/4)\log(1-1/p)+\log J_0(z/\sqrt{p}))$. Hence, by Merten's formula,
\begin{equation}\label{eq:logconv}
(-z^2/4)(\log\log x+\gamma)+\sum_{p\leq x}\log J_0\Big(\frac{-iz}{\sqrt{p}}\Big)
\rightarrow \log\Phi(-iz)\ \ \ as\ x\rightarrow\infty
\end{equation}
uniformly for $|z|\leq r$, $z\in\mathbb{C}$. The uniform convergence implies (see \cite[Theorem III.1.3]{FB}), that the $m$th derivative of the left hand side of~\eqref{eq:logconv} evaluated at $0$ converges to $c_m$. Hence, under the assumptions of Proposition~\ref{prop2}, the cumulants of  $\sum_{j=1}^{\pi(x)}\operatorname{Im}X_j/\sqrt{p_j}$ satisfy the convergence described in Corollary~\ref{cum}, since $\mathbb{E}[\exp(z\sum_{j=1}^{\pi(x)}\operatorname{Im}X_j/\sqrt{p_j})]=\prod_{p\leq x}J_0(-iz/\sqrt{p})$. It remains to show that for $m\geq 1$
\begin{equation}\label{eq:cumc}
 \kappa_m\left( \operatorname{Im}\Sigma_{1,x}(-U_T)\right) -\kappa_m\bigg(\sum_{j=1}^{\pi(x)} \frac{\operatorname{Im}X_j}{\sqrt{p_j}}\bigg)\rightarrow 0\ \ \ \ \ \ as\ T\rightarrow\infty.
\end{equation}
To prove this, we use the fact that the cumulants can be expressed in terms of the moments, namely $\kappa_m(Y)=\sum a_{\lambda_1,\dots,\lambda_m}\mathbb{E}[Y^1]^{\lambda_1}\cdots \mathbb{E}[Y^m]^{\lambda_m}$, where the sum is over all positive integers such that $1\lambda_1+2\lambda_2+\cdots+m\lambda_m=m$, $a_{\lambda_1,\dots,\lambda_m}$ are integers, and $Y$ is a random variable as above. If we plug in Proposition~\ref{prop1} and~\eqref{eq:ism} into this formula,~\eqref{eq:cumc} follows from multiplying out since for $k\leq m$ and $x\geq 3$ the main terms in~\eqref{eq:im} are $O((\sum_{p\leq x}1/p)^m)=O((\log\log x)^m)$ (see \cite[(5) of chapter 7]{D}), while for $k\leq m$ the remainders in~\eqref{eq:im} are $O(T^{(m/N)-1})$ which is $O(T^{-a})$ for some $0<a<1$ if $T$ is sufficiently large. 
\end{proof}

\section{Mod-convergence of a sum over primes}\label{mcsp}

By means of Proposition~\ref{prop1} and~\eqref{eq:ism}, we can apply  the method of moments for fixed \(x\) and obtain the following convergence
\begin{equation}\label{eq:lr}
  \frac{1}{T}\int_{T}^{2T} e^{iu\sum_{p\leq x}\frac{\sin (t\log p)}{\sqrt{p}}}dt\rightarrow
  \prod_{p\leq x}J_0\Big(\frac{u}{\sqrt{p}}\Big)\ \ \ as\ T\rightarrow\infty.
\end{equation}
Another proof of~\eqref{eq:lr} is contained in \cite[Theorem 5.1]{L}.
The techniques used therein can be applied to get Theorem~\ref{thma} and Theorem~\ref{thmb} for the choice \(x=(\log T)^{2-\epsilon}\), \(\epsilon>0\) arbitrary.
The improvement of Theorem~\ref{thma} follows from Proposition~\ref{propa} combined with Proposition~\ref{prop2}.

\begin{proposition}\label{propa} Let \(c>1\) be a constant. Define \(x=e^{\log T/N}\) with \(N=(c'ec^2/4)\log\log T\), where \(c'>1\) is allowed to depend on $T$ but such that $x\rightarrow\infty$ as $T\rightarrow\infty$. For $T\geq3$, sufficiently large such that $x\geq 2$ and $N\geq 1$, we have
\begin{equation}\label{eq:p1}
 \frac{1}{T}\int_{T}^{2T} e^{iu\sum_{p\leq x}\frac{\sin (t\log p)}{\sqrt{p}}}dt
 =\prod_{p\leq x}J_0\Big(\frac{u}{\sqrt{p}}\Big)+O((1/c')^{N-1}+(2c^2/\log x)^N)
\end{equation}
uniformly for \(|u|\leq c\), $u\in\mathbb{R}$.
\end{proposition}

\begin{proof}[Proof of Theorem~\ref{thma}]
We apply Proposition~\ref{propa} with $x$, $N$ as in Theorem~\ref{thma} and $c$ an arbitrary constant with $c>1$. Since $c'\rightarrow\infty$ in that case, the remainder in~\eqref{eq:p1} is $o(\exp(-c^2(\log\log T)/4))$. If we multiply in~\eqref{eq:p1} both sides by $\exp(u^2(\log\log x+\gamma)/4)$ and then apply Proposition~\ref{prop2}, we obtain~\eqref{eq:i3} uniformly for $|u|\leq c$. Since $c>1$ is arbitrary, this completes the proof.
\end{proof}

\begin{proof}[Proof of Proposition~\ref{propa}] Let $N'=\left\lfloor N\right\rfloor$. From the Taylor expansion $e^{iu}=\sum_{k\leq 2N'-1}(iu)^k/k!+$ $\theta u^{2N'}/(2N')!$, $u\in\mathbb{R}$, with $|\theta|\leq 1$, we obtain
\begin{multline}\label{eq:eqp1}
  \frac{1}{T}\int_{T}^{2T} e^{iu\sum_{p\leq x}\frac{\sin (t\log p)}{\sqrt{p}}}dt  = 
  \sum_{k\leq 2N'-1}\frac{(iu)^k}{k!}\frac{1}{T}\int_{T}^{2T}\bigg(\sum_{p\leq x}\frac{\sin (t\log p)}{\sqrt{p}}\bigg)^kdt\\
  + \theta\frac{u^{2N'}}{(2N')!}\frac{1}{T}\int_{T}^{2T}
  \bigg(\sum_{{p\leq x}}{\frac{\sin(t\log p)}{\sqrt{p}}}\bigg)^{2N'}dt
\end{multline}
with \(|\theta|\leq 1\). By Proposition~\ref{prop1}, the remainder is 
\begin{equation*}
  O\bigg(\frac{c^{2N'}}{N'!}\frac{1}{2^{2N'}}\Big(\sum_{p\leq x}\frac{1}{p}\Big)^{N'}+\frac{(c^2\pi(x))^{N'}}{T}\bigg).
\end{equation*}
Using the bound $(N')!\geq (N'/e)^{N'}$, elementary results in the theory of primes, namely the formulas $\sum_{p\leq x}1/p=\log\log x+c_1+O(1/\log x)$ and $\pi(x)\leq 2 x/\log x$, and finally $N'=\left\lfloor N\right\rfloor$, this is 
\begin{equation*}
  O\bigg(\Big(\frac{ec^2\log\log T}{4N'}\Big)^{N'}+\frac{(c^2\pi(x))^N}{T}\bigg)= O\bigg(\Big(\frac{1}{c'}\Big)^{N-1}+\Big(\frac{2c^2}{\log x}\Big)^N\bigg).
\end{equation*}

Now, let $X_1,X_2,\dots$ be an i.i.d. sequence of random variables uniformly distributed on the unit circle. By Proposition~\ref{prop1} and~\eqref{eq:ism}, the moments in~\eqref{eq:eqp1} are equal to those of the stochastic model plus a remainder which is bounded by $(2D/T)\sqrt{(\pi(x))^kk!}$. The resulting remainders in~\eqref{eq:eqp1}, $k\leq 2N'-1$, add up to $O((c^2\pi(x))^{N}/T)=O((2c^2/\log x)^N)$. Hence,~\eqref{eq:eqp1} is equal to 
\begin{equation*}
 \sum_{k\leq 2N'-1}\frac{(iu)^k}{k!}
 \mathbb{E}\bigg[\bigg(\sum_{j=1}^{\pi(x)}\frac{\operatorname{Im}X_j}{\sqrt{p_j}}\bigg)^k\bigg]+
 O((1/c')^{N-1}+(2c^2/\log x)^N).
\end{equation*}
Applying the above Taylor expansion again, we obtain
\begin{multline*}
 \prod_{p\leq x}J_0\Big(\frac{u}{\sqrt{p}}\Big)=
 \mathbb{E}\bigg[e^{iu\sum_{j=1}^{\pi(x)}\frac{\operatorname{Im}X_j}{\sqrt{p_j}}}\bigg]\\ 
 = \sum_{k\leq 2N'-1}\frac{(iu)^k}{k!}\mathbb{E}\bigg[\bigg(\sum_{j=1}^{\pi(x)}\frac{\operatorname{Im}X_j}{\sqrt{p_j}}\bigg)^k\bigg]+ \theta\frac{u^{2N'}}{(2N')!}\mathbb{E}\bigg[\bigg(\sum_{j=1}^{\pi(x)}\frac{\operatorname{Im}X_j}{\sqrt{p_j}}\bigg)^{2N'}\bigg]
\end{multline*}
with $|\theta|\leq 1$. The remainder already appeared in~\eqref{eq:eqp1} and is $O((1/c')^{N-1})$. This completes the proof.
\end{proof}

\section{Mod-convergence in the complex plane}\label{mccp}

Section~\ref{mccp} is devoted to the proof of Theorem~\ref{thmb}. Here, we will apply an explicit formula obtained by Goldston \cite[Lemma 1]{G} assuming RH. For $4\leq x\leq t^2$ and $t\neq\gamma$, we have
\begin{multline}
 \operatorname{Im}\log\zeta(1/2+it)=-\sum_{n\leq x}\frac{\Lambda(n)}{\log n}\frac{\sin(t\log n)}{\sqrt{n}}f\bigg(\frac{\log n}{\log x}\bigg)\\
 +\sum_{\gamma}\sin((t-\gamma)\log x)\int_0^\infty\frac{u}{u^2+((t-\gamma)\log x)^2}\frac{du}{\sinh u}+O\bigg(\frac{1}{t(\log x)^2}\bigg)\label{eq:G},
\end{multline}
where $f(u)=(\pi u/2)\cot(\pi u/2)$. We will also apply the following estimate obtained by Soundararajan assuming RH. For every $h\in\mathbb{R}$ there exist constants $C',C''>0$ such that 
\begin{equation}\label{eq:c1}
 \frac{1}{T}\int_T^{2T} e^{h\operatorname{Im}\log\zeta(1/2+it)}dt\leq C''e^{C'h^2\log\log T}.
\end{equation}
Soundararajan \cite{SO} proved~\eqref{eq:c1} for $\operatorname{Re}\log\zeta(1/2+it)$. However, by using \cite[Theorem 1]{S} instead of~\cite[Proposition]{SO}, his arguments apply to $\operatorname{Im}\log\zeta(1/2+it)$, too. We prove (compare to \cite[Lemma 3 and Corallary]{BH}):

\begin{proposition}\label{propc} Assume RH. Let $4\leq x\leq T^2$, $f(u)=(\pi u/2)\cot(\pi u/2)$. For every \(h\in\mathbb{R}\) there exist constants $C,C'$, and $C''$ such that
\begin{equation*}
 \frac{1}{T}\int_T^{2T} e^{h\sum_{n\leq x}\frac{\Lambda(n)}{\log n}\frac{\sin(t\log n)}{\sqrt{n}}f
 \left(\frac{\log n}{\log x}\right)}dt\leq C''e^{C|h|\frac{\log T}{\log x}+C'h^2\log\log T}.
\end{equation*}
\end{proposition}

\begin{proof}[Proof of Theorem~\ref{thmb}]
The proof mainly differs from the proof of Theorem~\ref{thma} and Proposition~\ref{propa} in its estimation of the remainder term. Nevertheless, we will repeat the main steps. We assume that $|z|\leq c$, where $c>1$ is an arbitrary constant and $z\in\mathbb{C}$, say $z=u-ih$ with $u,h\in\mathbb{R}$. Let $N'=\left\lfloor N/2\right\rfloor$. From the Taylor expansion $e^{iz}=e^{h+iu}=\sum_{k\leq 2N'-1}(iz)^k/k!+\theta e^h (|z|^{2N'}/(2N')!)$ with $|\theta|\leq 1$, we obtain
\begin{multline}\label{eq:ttcc}
  \frac{1}{T}\int_{T}^{2T} e^{iz\operatorname{Im}\Sigma_{f,x}
  (-t)}dt =\sum_{k\leq 2N'-1}\frac{(iz)^k}{k!}
  \frac{1}{T}\int_{T}^{2T}\left( \operatorname{Im}
  \Sigma_{f,x}(-t)\right) ^kdt\\
  + \theta\frac{c^{2N'}}{(2N')!}\frac{1}{T}\int_{T}^{2T}
  e^{h\operatorname{Im}\Sigma_{f,x}(-t)}\left( \operatorname{Im}\Sigma_{f,x}(-t)\right) ^{2N'}dt
\end{multline}
with $|\theta|\leq 1$. To continue as in the proof of Proposition~\ref{propa}, we use that under the assumptions of Proposition~\ref{prop1} we have
\begin{equation}\label{eq:versionm}
 \frac{1}{T}\int_T^{2T}(\operatorname{Im}\Sigma_{f,x}(-t))^kdt=
 \mathbb{E}\bigg[\bigg(\sum_{j=1}^{\pi(x)}\frac{\operatorname{Im}X_j}{\sqrt{p_j}}f\bigg(\frac{\log p_j}{\log x\ }\bigg)\bigg)^{k}\bigg]+\theta\frac{2D}{T}\sqrt{(\pi(x))^kk!}
\end{equation}
with $|\theta|\leq 1$. Moreover, if we replace $k$ by $2k$, the main term is bounded by $((2k)!/2^{2k}k!)(\sum_{p\leq x}1/p)^k$. These estimates follow as in the proof of Proposition~\ref{prop1} and~\eqref{eq:ism}.
Applying the Cauchy-Schwarz inequality, the absolut value of the remainder in~\eqref{eq:ttcc} can be bounded by
\begin{equation*}
 \frac{c^{2N'}}{(2N')!}\sqrt{\frac{1}{T}\int_{T}^{2T}(\operatorname{Im}\Sigma_{f,x}(-t))^{4N'}dt}
 \sqrt{\frac{1}{T}\int_{T}^{2T}e^{2h\operatorname{Im}\Sigma_{f,x}(-t)}dt}.
\end{equation*}
For $x\geq 2$, we have $|\operatorname{Im}\Sigma_{f,x}(-t)-\operatorname{Im}\Sigma^*_{f,x}(-t)|\leq(\log\log x)/2+O(1)$, say $\leq CN$ for $T$ sufficiently large. Hence, by~\eqref{eq:versionm} and Proposition~\ref{propc}, the above is
\begin{equation*}
 O\left(\frac{c^{2N'}}{(2N')!}\sqrt{\frac{(4N')!(\sum_{p\leq x}1/p)^{2N'}}{2^{4N'}(2N')!}+\frac{\sqrt{(4N')!(\pi(x))^{4N'}}}{T}}\sqrt{e^{4CcN+4C'c^2N}}\right)
\end{equation*} 
for $T$ sufficiently large.
Applying $(4N')!/(2^{4N'}(2N')!)\leq (2N')!$, $\sqrt{(2N')!}\geq (2N'/e)^{N'}$, $\pi(x)\leq 2x/\log x$, and $N'=\left\lfloor N/2\right\rfloor$, there exists a constant $c''>0$ (depending on $c$, $C$, and $C'$) such that this is 
\begin{equation*}
 O\left(\left(\frac{c''\sum_{p\leq x}1/p}{N'}\right)^{N'}+\left(\frac{c''}{\log x}\right)^{N/2}\right).
\end{equation*}
Since $N'/\sum_{p\leq x}1/p$ and $\log x$ go to infinity, this is $o(\exp(-c^2(\log\log x)/4))$. Hence, applying~\eqref{eq:versionm} to the other terms,~\eqref{eq:ttcc} is equal to
\begin{equation*}
\sum_{k\leq 2N'-1}\frac{(iz)^k}{k!}\mathbb{E}\bigg[\bigg(\sum_{j=1}^{\pi(x)}\frac{\operatorname{Im}X_j}{\sqrt{p_j}}f\bigg(\frac{\log p_j}{\log x\ }\bigg)\bigg)^{k}\bigg]+o(e^{-c^2(\log\log T)/4})
\end{equation*}
uniformly for $|z|\leq c$.
If we replace $\sin(t\log p)$ by $\operatorname{Im}X_i$ and integration by expectation in~\eqref{eq:ttcc}, we can bound the resulting remainder as above. In doing so, we apply~\eqref{eq:modconv} instead of Proposition~\ref{propc}. The result is that
\begin{equation*}
 \frac{1}{T}\int_{T}^{2T} e^{iz\operatorname{Im}\Sigma_{f,x}(-t)}dt=\mathbb{E}\bigg[e^{iz\sum_{j=1}^{\pi(x)}\frac{\operatorname{Im}X_j}{\sqrt{p_j}}f\left( \frac{\log p_j}{\log x\ }\right) }\bigg]\\+o(e^{-c^2(\log\log T)/4})
\end{equation*}
uniformly for $|z|\leq c$. If we multiply both sides by $\exp(z^2(\log\log x+\gamma_f)/4)$ and then apply the formula
\begin{equation}\label{eq:modconv}
  e^{z^2(\log\log x+\gamma_f)/4}\mathbb{E}  
  \Big[e^{iz\sum_{j=1}^{\pi(x)}\frac{\operatorname{Im}X_j}{\sqrt{p_j}}
  f\left(\frac{\log p_j}{\log x\ }\right)}\Big]\rightarrow 
  \Phi(z)\ \ \ \ \ \ as\ x\rightarrow\infty
\end{equation}
locally uniformly for $z\in\mathbb{C}$, the statement of Theorem~\ref{thmb} follows by the same argument as in the proof of Theorem~\ref{thma}. ~\eqref{eq:modconv} follows as in the proof of Proposition~\ref{prop2} by using the additional fact, that
\begin{equation*}
\prod_{p\leq x}\bigg(1-\frac{1}{p}f^2\bigg(\frac{\log p}
{\log x }\bigg)\bigg)^{-z^2/4}J_0\bigg(\frac{z}{\sqrt{p}}f\bigg(\frac{\log p}{\log x }\bigg)\bigg)\rightarrow\Phi(z)\ \ \ as\ x\rightarrow\infty
\end{equation*}
locally uniformly for $z\in\mathbb{C}$. We conclude by a brief argument why this holds. Split the product in $p\leq y$ and $y<p\leq x$. The product over $y<p\leq x$ converges locally uniformly to $1$ if $y\rightarrow\infty$, while one can show that the product over $p\leq y$, say $y=\log x$, converges locally uniformly to $\phi(z)$, by using, e.g., $f^2(\log p/\log x)-1=O((\log\log x)/\log x)$ if $p\leq \log x$. This completes the proof.
\end{proof}

\begin{proof}[Proof of Proposition~\ref{propc}] From the formulas \eqref{eq:G},~\eqref{eq:c1} and the Cauchy-Schwarz inequality, we obtain
\begin{multline*}
  \frac{1}{T}\int_T^{2T} e^{h\sum_{n\leq x}\frac{\Lambda(n)}{\log n}\frac{\sin(t\log n)}{\sqrt{n}}f
  \left(\frac{\log n}{\log x}\right)}dt\\
  \leq C'''e^{2C'h^2\log\log T}\sqrt{\frac{1}{T}\int_T^{2T} e^{2h\sum_{\gamma}\sin((t-\gamma)\log x)
  \int_0^\infty\frac{u}{u^2+((t-\gamma)\log x)^2}\frac{du}{\sinh u}}dt}
\end{multline*}
where $C'''$ is a constant. The absolute value of the sum over zeros is bounded by a constant times
\begin{equation}\label{eq:1}
 \sum_{|(t-\gamma)\log x|\leq 1}1+\sum_{|(t-\gamma)\log x|>1}\frac{1}{((t-\gamma)\log x)^2}
\end{equation}
and therefore it suffices to deal with the exponential moments of~\eqref{eq:1} with $h\geq 0$. Using the Cauchy-Schwarz inequality again, we obtain
\begin{multline*}
 \frac{1}{T}\int_T^{2T} e^{h\sum_{|(t-\gamma)\log x|\leq 1}1+h\sum_{|(t-\gamma)\log x|>1}\frac{1}{((t-\gamma)\log x)^2}}dt\\
 \leq\sqrt{\frac{1}{T}\int_T^{2T} e^{2h\sum_{|(t-\gamma)\log x|\leq 1} 1}dt}
 \sqrt{\frac{1}{T}\int_T^{2T} e^{2h\sum_{|(t-\gamma)\log x|>1}\frac{1}{((t-\gamma)\log x)^2}}dt}.
\end{multline*}
We start with the first term, using the following fact on the number of zeros (see \cite[(1) of Ch. 15]{D})
\begin{equation}\label{eq:nz}
 N(t)=\frac{t}{2\pi}\log\frac{t}{2\pi}-\frac{t}{2\pi}+\frac{7}{8}+S(t)+O(1/t),
\end{equation}
where $t\neq\gamma$ and $S(t)=(1/\pi)\operatorname{Im}\log\zeta(1/2+it)$. We compute (note that $h\geq 0$)
\begin{align}
 \frac{1}{T}\int_T^{2T} &e^{h\sum_{|(t-\gamma)\log x|\leq 1} 1}dt\tag*{} \\
 &=\frac{1}{T}\int_T^{2T} e^{h\left(N\left(t+\frac{1}{\log x}\right)
 -N\left(t-\frac{1}{\log x}\right)\right)}dt\tag*{}\\
 &\leq\frac{1}{T}\int_T^{2T} e^{Ch\frac{\log t}{\log x}+h\left(S\left(t+\frac{1}{\log x}\right)-S\left(t-\frac{1}{\log x}\right)\right)}dt \tag*{}\\
 &\leq e^{Ch\frac{\log T}{\log x}}\sqrt{\frac{1}{T}\int_T^{2T} e^{2hS\left(t+\frac{1}{\log x}\right)}dt}
 \sqrt{\frac{1}{T}\int_T^{2T} e^{-2hS\left(t-\frac{1}{\log x}\right)}dt}\tag*{}\\
 &=O(e^{Ch\frac{\log T}{\log x}+4C'(h/\pi)^2\log\log T})\label{eq:www}.
\end{align}
In the last step we used~\eqref{eq:c1}.
Next, we divide the sum over $|(t-\gamma)\log x|>1$ into $|t-\gamma|\geq T$, $1<|t-\gamma|<T$, and $1/\log x<|t-\gamma|\leq 1$. 

For $t\in[T,2T]$, we have
\begin{equation*}
 \sum_{|t-\gamma|\geq T}\frac{1}{((t-\gamma)\log x)^2}=O\bigg(\sum_\gamma\frac{1}{\gamma^2(\log x)^2}\bigg)=O\bigg(\frac{1}{(\log x)^2}\bigg).
\end{equation*}
The last step results from \cite[(4) of Ch. 12]{D}.
For the second sum we use the fact that $N(t+1)-N(t)=O(1+\log^+ |t|)$ (see \cite[(2) of Ch. 15]{D}). For $t\in[T,2T]$, we obtain
\begin{align*}
 &\sum_{1<|t-\gamma|<T}\frac{1}{((t-\gamma)\log x)^2}\\
 &\ \ \ \leq\sum_{k=1}^{\left\lceil T\right\rceil-1}\frac{N(t+k+1)-N(t+k)}{k^2(\log x)^2}
 + \sum_{k=1}^{\left\lceil T\right\rceil-1}\frac{N(t-k)-N(t-k-1)}{k^2(\log x)^2}\\
 &\ \ \ =O\bigg(\sum_{k=1}^{\left\lceil T\right\rceil-1}\frac{\log T}{k^2(\log x)^2}\bigg)=O\bigg(\frac{\log T}{(\log x)^2}\bigg).
\end{align*}
Next, we consider the sum over $1/\log x<\gamma-t\leq 1$. 
We have
\begin{equation}\label{eq:2}
 \sum_{1/\log x<\gamma-t\leq 1}\frac{1}{((t-\gamma)\log x)^2}\leq\sum_{j=1}^{M}\frac{N\left(t+\frac{k_j}{\log x}\right)-
 N\left(t+\frac{k_{j-1}}{\log x}\right)}{k_{j-1}^2}
\end{equation}
where $1=k_0<k_1<\cdots{}<k_M$ with $k_{M-1}<\log x\leq k_M$. By~\eqref{eq:nz}, this is bounded by, recall $t\in[T,2T]$,
\begin{equation*}
 \sum_{j=1}^{M}\left(\frac{C(k_j-k_{j-1})\log T}{(\log x)k_{j-1}^2}+\frac{S\left(t+\frac{k_j}{\log x}\right)
 -S\left(t+\frac{k_{j-1}}{\log x}\right)}{k_{j-1}^2}\right).
\end{equation*}
We choose $k_j=2^{j/2}$ and bound the left hand side of~\eqref{eq:2} by 
\begin{equation*}
 \sqrt{2}C\frac{\log T}{\log x}+\sum_{j=1}^{M}\frac{S\left(t+\frac{2^{j/2}}{\log x}\right)-S\left(t+\frac{2^{(j-1)/2}}{\log x}\right)}{2^{j-1}}.
\end{equation*}
It follows that 
\begin{align*}
 \frac{1}{T}&\int_T^{2T} e^{h\sum_{1/\log x<\gamma-t\leq 1}\frac{1}{((t-\gamma)\log x)^2}}dt\\
 &\leq e^{\sqrt{2}Ch\frac{\log T}{\log x}}\frac{1}{T}\int_T^{2T} e^{h\sum_{j=1}^{M}\frac{1}{2^{j-1}}\left(S\left(t+\frac{2^{j/2}}{\log x}\right)
 -S\left(t+\frac{2^{(j-1)/2}}
 {\log x}\right)\right)}dt.\hspace{1cm}
\end{align*}
Using $\mathbb{E}[e^{h\sum_{j=1}^MX_j/2^j}]\leq\prod_{j=1}^{M}(\mathbb{E}[e^{hX_j}])^{1/2^j}$, which follows from repeated application of the Cauchy-Schwarz inequality, this is
\begin{equation*}
\leq e^{\sqrt{2}Ch\frac{\log T}{\log x}}\prod_{j=1}^M\left(\frac{1}{T}\int_T^{2T} e^{2h\left(S\left(t+\frac{2^{j/2}}{\log x}\right)-S\left(t+\frac{2^{(j-1)/2}}{\log x}\right)\right)}dt\right)^{1/2^j}.
\end{equation*}
Applying again the Cauchy-Schwarz inequality and then~\eqref{eq:c1} (as in~\eqref{eq:www}), this is
\begin{equation*}
O(e^{\sqrt{2}Ch\frac{\log T}{\log x}}e^{16C'(h/\pi)^2\log\log T}).
\end{equation*}
The same bound is true for the sum over $1/\log x<t-\gamma\leq 1$. The claim now follows from putting together all these estimates.
\end{proof}

\section{Proof of Corollary~\ref{cor1}}
Let $T$, $c$, $c'$, $x$, and $N$ be as in Proposition~\ref{propa}, $T\geq 3$ sufficiently large such that $x\geq 2$ and $N\geq 2$. Assume further that $c'>4$ is a constant such that the bound $(\log\log T)^{1/2}(c'/4)^{-N/2}=O(1/\log T)$ holds and that $T$ is so big that the bound $(\log T)(2c^2/\log x)^{N/2}=O(1/\log T)$ holds, too. Then we show that  
\begin{align} \label{eq:st}
  &\frac{1}{T}\int_{T}^{2T} e^{iu\operatorname{Im}\log\zeta(1/2+it)}dt=\prod_{p\leq x}J_0\Big(\frac{u}{\sqrt{p}}\Big)\nonumber\\
  &-\sum_{\substack{p\leq x\\k\geq 3\ odd}}\frac{u}{k\sqrt{p}^k}J_k\Big(\frac{u}{\sqrt{p}}\Big)
 \prod_{\substack{q\leq x\\ q\neq p}}J_0\Big(\frac{u}{\sqrt{q}}\Big)
  +u^2O(\log\log\log T)+O(1/\log T)
\end{align} 
uniformly for $|u|\leq c$, $u\in\mathbb{R}$. One can deduce Corollary~\ref{cor1} from~\eqref{eq:st} as follows. Replace $u$ by $v/\sqrt{(\log\log T)/2}$ with $|v|\leq\sqrt{\log\log T/\log\log\log T}$ and let $T$ be sufficiently large. Then, by~\eqref{eq:infs}, the formula $\sum_{p\leq x}1/p=\log\log x+c_1+O(1/\log x)$, and $\log\log x/\log\log T=1+O(\log\log\log T/\log\log T)$, the first term on the right hand side of~\eqref{eq:st} is equal to
\begin{equation*}
\exp\bigg(\sum_{p\leq x}\log J_0(v/\sqrt{p(\log\log T)/2})\bigg)=e^{-v^2/2}\bigg(1+v^2O\bigg(\frac{\log\log\log T}{\log\log T}\bigg)\bigg)
\end{equation*}
and, by using $|J_k(u)|\leq (|u|/2)^{k}/k!$ and $|J_0(u)|\leq 1$, $u\in\mathbb{R}$, the second term is $v^4O(1/(\log\log T)^2)$ which is smaller than $v^2O(\log\log\log T/\log\log T)$. 

Hence, it remains to prove~\eqref{eq:st}. From $\operatorname{Im}\log\zeta(1/2+it)=\operatorname{Im}\Sigma_{1,x}(t)+\operatorname{Im}r_{1,x}(t)$ and Taylor's theorem, we obtain
\begin{multline*}
 \frac{1}{T}\int_{T}^{2T} e^{iu\operatorname{Im}\log\zeta(1/2+it)}dt=\frac{1}{T}\int_T^{2T}e^{iu\operatorname{Im}\Sigma_{1,x}(t)}dt\\+iu\frac{1}{T}\int_T^{2T}\operatorname{Im}r_{1,x}(t)e^{iu\operatorname{Im}\Sigma_{1,x}(t)}dt+\theta\frac{u^2}{2}\frac{1}{T}\int_T^{2T}(\operatorname{Im}r_{1,x}(t))^2dt
\end{multline*}
with $|\theta|\leq 1$. By Proposition~\ref{propa} and the above assumptions, the first term is equal to $\prod_{p\leq x}J_0(u/\sqrt{p})$ $+O(1/\log T)$ and by \cite[Corollary of Theorem 5.1]{T}, the third term is $u^2O(\log\log\log T)$. It remains to consider the second term. We start showing that
\begin{multline}\label{eq:ims}
 \frac{1}{T}\int_{T}^{2T}\operatorname{Im}\log\zeta(1/2+it)e^{iu\operatorname{Im}\Sigma_{1,x}(t)}dt
 \\=\sum_{\substack{p\leq x\\k\geq 1\ odd}}\frac{i}{k\sqrt{p}^k}J_k\Big(\frac{u}{\sqrt{p}}\Big)
 \prod_{\substack{q\leq x\\ q\neq p}}
 J_0\Big(\frac{u}{\sqrt{q}}\Big)+O(1/\log T)
\end{multline}
uniformly for \(|u|\leq c\). Let $N'=\left\lfloor N/2\right\rfloor$. From the Taylor expansion $e^{iu}=\sum_{k\leq 2N'-1}(iu)^k/k!+$ $\theta u^{2N'}/(2N')!$, $u\in\mathbb{R}$, with $|\theta|\leq 1$, we obtain that the left hand side of~\eqref{eq:ims} is equal to
\begin{multline}\label{eq:ttt}
 \sum_{k\leq 2N'-1}\frac{(iu)^k}{k!}\frac{1}{T}\int_{T}^{2T}\operatorname{Im}\log\zeta(1/2+it)
 (\operatorname{Im}\Sigma_{1,x}(t))^kdt\\  
 +\theta\frac{c^{2N'}}{(2N')!}\frac{1}{T}\int_{T}^{2T}|\operatorname{Im}\log\zeta(1/2+it)|(\operatorname{Im}\Sigma_{1,x}(t))^{2N'}dt
\end{multline}
with \(|\theta|\leq 1\). Applying the Cauchy-Schwarz inequality, the estimates in the proof of Proposition~\ref{propa}, and \cite[Theorem 3]{S2}, i.e. $(1/T)\int_T^{2T}(\operatorname{Im}\log\zeta(1/2+it))^2dt=(\log\log T)/2+O(1)$, the remainder is $O((\log\log T)^{1/2}((c'/4)^{-N/2}+(2c^2/\log x)^{N/2})=O(1/\log T) $. The remaining moments can be computed by using the following lemma which is a modification of \cite[Lemma 5]{S} and \cite[equation (6.3)]{G} and serves as a substitute for the mean value theorem of Montgomery and Vaughan in Section~\ref{moments}.
\begin{lemma} Assume RH. Let \(k,h\leq T\) be two positive integers with \((k,h)=1\). Then
\begin{align}
  \hspace{6mm}\int_T^{2T}\log\zeta(1/2+it)\Big(\frac{k}{h}\Big)^{it}dt=
  \frac{T\Lambda(k)}{\sqrt{k}\log k}
  +&O(\sqrt{kh}\log T),\ h=1\nonumber\\
   &O(\sqrt{kh}\log T),\ h\neq1, \nonumber\\
  \int_T^{2T}\operatorname{Im}\log\zeta(1/2+it)\Big(\frac{k}{h}\Big)^{it}dt=\frac{-iT\Lambda(k)}{2\sqrt{k}\log k}
  +&O(\sqrt{kh}\log T),\ h=1\label{eq:eqtsg}\\
    =\frac{iT\Lambda(h)}{2\sqrt{h}\log h}
  +&O(\sqrt{kh}\log T),\ k=1\nonumber\\
   &O(\sqrt{kh}\log T),\ h,k\neq1. \nonumber
\end{align}
\end{lemma}

Denote by \(p_1,p_2,\dots,p_n\) the prime numbers not exceeding \(x\), and let $X_1,X_2,\dots$ be an i.i.d. sequence of random variables uniformly distributed on the unit circle. Furthermore, let $k,h\leq T$ be positive integers with $k/h=p_1^{-k_1}\cdots p_n^{-k_n}$. Then~\eqref{eq:eqtsg} can be written as
\begin{align}
 &\frac{1}{T}\int_T^{2T}\operatorname{Im}\log\zeta(1/2+it)(p_1^{-k_1}\cdots p_n^{-k_n})^{it}dt\label{eq:lem}\\
 =&\mathbb{E}\bigg[ -\sum_{j=1}^n\operatorname{Im}\log(1-X_j/\sqrt{p_j})X_1^{k_1}\cdots X_n^{k_n}\bigg]+O\Big(\frac{1}{T}\sqrt{p_1^{|k_1|}\cdots p_n^{|k_n|}}\log T\Big)\nonumber.
\end{align}
Expanding $(\operatorname{Im}\Sigma_{1,x}(t))^k$ as in~\eqref{eq:binth} and~\eqref{eq:multth}, we deduce from~\eqref{eq:lem} that 
\begin{multline*}
\frac{1}{T}\int_{T}^{2T}\operatorname{Im}\log\zeta(1/2+it)
 (\operatorname{Im}\Sigma_{1,x}(t))^kdt\\
 =\mathbb{E}\bigg[\bigg( -\sum_{j=1}^n\operatorname{Im}\log(1-X_j/\sqrt{p_j})\bigg)
\bigg(\sum_{j=1}^n\frac{\operatorname{Im}X_j}{\sqrt{p_j}}\bigg)^k\bigg]\\
+O\bigg(\frac{\log T}{2^{k}T}\sum_{l=0}^{k}\binom{k}{l}
  \sum_{\lambda_1+\cdots+\lambda_n=l}\frac{l!}{\lambda_1!\cdots\lambda_n!}\sum_{\lambda_1+\cdots+\lambda_n=k-l}\frac{(k-l)!}{\lambda_1!\cdots\lambda_n!}\bigg).
\end{multline*}
The remainder is $O((\log T) n^k/T)$ and the resulting remainders in~\eqref{eq:ttt}, $k\leq 2N'-1$, add up to 
$O((\log T)(2c/\log x)^{N})=O(1/\log T)$. Hence,~\eqref{eq:ttt} is equal to
\begin{multline}\label{eq:bilog}
 \sum_{k\leq 2N'-1}\frac{(iu)^k}{k!}\mathbb{E}\bigg[\bigg( -\sum_{j=1}^n\operatorname{Im}\log(1-X_j/\sqrt{p_j})\bigg)
\bigg(\sum_{j=1}^n\frac{\operatorname{Im}X_j}{\sqrt{p_j}}\bigg)^k\bigg]
+O(1/\log T)\\
=\mathbb{E}\bigg[\bigg( -\sum_{j=1}^n\operatorname{Im}\log(1-X_j/\sqrt{p_j})\bigg)
e^{iu\sum_{j=1}^n\frac{\operatorname{Im}X_j}{\sqrt{p_j}}}\bigg]\\
+ \theta\frac{c^{2N'}}{(2N')!}\mathbb{E}\bigg[\Big| -\sum_{j=1}^n\operatorname{Im}\log(1-X_j/\sqrt{p_j})\Big|
\bigg(\sum_{j=1}^n\frac{\operatorname{Im}X_j}{\sqrt{p_j}}\bigg)^{2N'}\bigg]+O(1/\log T).
\end{multline}
The last equality follows from applying Taylor's theorem as in~\eqref{eq:ttt}. If one treats the first remainder in the last row as the corresponding one in~\eqref{eq:ttt}, using $\mathbb{E}[(\sum_{j=1}^{\pi(x)}\operatorname{Im}\log(1-X_j/\sqrt{p_j}))^2]=(\log\log x)/2+O(1)$ this time, one can show that it is also $O(1/\log T)$. By plugging in~\eqref{eq:bf} and expanding the logarithm, we obtain that~\eqref{eq:bilog} is equal to
\begin{equation*}
\sum_{\substack{p\leq x\\k\geq1\ odd}}\frac{i}{k\sqrt{p}^k}J_k\Big(\frac{u}{\sqrt{p}}\Big)
 \prod_{\substack{q\leq x\\ q\neq p}}J_0\Big(\frac{u}{\sqrt{q}}\Big)+O(1/\log T)
\end{equation*}
which completes the proof of~\eqref{eq:ims}. The last step in the proof of~\eqref{eq:st} is to show that
\begin{align*}
 \frac{1}{T}\int_{T}^{2T}\operatorname{Im}&\Sigma_{1,x}(t)e^{iu\operatorname{Im}\Sigma_{1,x}(t)}dt\\&
 =\mathbb{E}\bigg[\bigg(\sum_{j=1}^{n}\frac{\operatorname{Im}X_j}
 {\sqrt{p_j}}\bigg)e^{iu\sum_{j=1}^{n}\frac{\operatorname{Im}X_j}
 {\sqrt{p_j}}}\bigg]+O(1/\log T)\\
 &=\sum_{p\leq x}\frac{i}{\sqrt{p}}J_1\Big(\frac{u}{\sqrt{p}}\Big)\prod_{\substack{q\leq x\\ q\neq p}}J_0\Big(\frac{u}{\sqrt{q}}\Big)+O(1/\log T)
\end{align*}
uniformly for \(|u|\leq c\). The first equality follows as above or as in the proof of Proposition~\ref{prop1}, the second equality again by plugging in~\eqref{eq:bf}. This completes the proof.\qed

\section{Proof of Corollary~\ref{cor2} and~\ref{cor3}} 

\begin{proof}[Proof of Corollary~\ref{cor2}]
Let $x\geq 2$ be as in Theorem~\ref{thmb} with the additional property that $N/\log\log T=O(\log\log T)$. By Theorem~\ref{thmb} and the fact that $\log\log T/\log\log x\rightarrow 1$ in this case, we obtain for each $h\in\mathbb{R}$
\begin{equation}
 \frac{1}{(\log\log T)/2}\log\bigg(\frac{1}{T}\int_T^{2T}e^{h\operatorname{Im}\Sigma_{f,x}(t)}dt\bigg)\rightarrow h^2/2\ \ \ \ \ as\ T\rightarrow\infty.
\end{equation}
By Theorem~\ref{tc1}, we obtain that
the family $(1/((\log\log T)/2))\operatorname{Im}\Sigma_{f,x}(U_T)$ satisfies the large deviation principle with the speed $1/((\log\log T)/2)$ and the rate function $I(h)=h^2/2$. Next, consider $\operatorname{Im}r_{f,x}(U_T)$. We will show that there exists a constant $C>0$ (the constant in~\eqref{eq:selb}) such that for each $\delta>0$ 
\begin{multline}\label{eq:exb}
(1/T)\lambda(\{t\in[T,2T]:|\operatorname{Im}r_{f,x}(t)|\geq C\delta\log\log T\})\\\leq e^{-(1-o(1))(\delta\log\log T)\log (\delta\log\log T)}.
\end{multline}
We postpone the proof of~\eqref{eq:exb} to the end of this section. From~\eqref{eq:exb} we deduce that for each $\delta>0$
\begin{multline*}
\frac{1}{(\log\log T)/2}\log\bigg(\frac{1}{T}\lambda(\{t\in[T,2T]:|\operatorname{Im}r_{f,x}(t)|\geq \delta \log\log T\})\bigg)\\
\leq -2(\delta/C)(1-o(1))(\log\log\log T+\log (\delta/C)).
\end{multline*}
As $T\rightarrow\infty$, the right hand side goes to $-\infty$. Hence, by Definition~\ref{exequiv}, the families $(1/((\log\log T)/2))\operatorname{Im}\log\zeta(1/2+iU_T)$ and $(1/((\log\log T)/2))\Sigma_{f,x}(U_T)$ are exponentially equivalent. To obtain the statement of the theorem, we finally apply \cite[Theorem 4.2.13]{DZ}, which states that if two families of random variables are exponentially equivalent, and one of them satisfies the large deviation principle with good rate function $I$, then the same large deviation principle holds for the other family. 

It remains to show~\eqref{eq:exb}. Therefore, let $V=\delta\log\log T$ and decompose
\begin{equation*}
\operatorname{Im}r_{f,x}=\operatorname{Im}(r_{g,T^{1/V}}^*
+(\Sigma_{g,T^{1/V}}^*-\Sigma_{g,T^{1/V}})
+(\Sigma_{g,T^{1/V}}-\Sigma_{g,x})
+\Sigma_{g-f,x}).
\end{equation*}
If $|\operatorname{Im}r_{f,x}(t)|\geq CV$, there exists a summand on the right hand side whose absolute value is greater or equal to $CV/4$. Applying the union bound, we obtain
\begin{align}
(1/T)&\lambda(\{t\in[T,2T]:|\operatorname{Im}r_{f,x}(t)|\geq CV\}\nonumber\\
\leq &(1/T)\lambda(\{t\in[T,2T]:|\operatorname{Im}r_{g,T^{1/V}}^*(t)|\geq CV/4\})\nonumber\\
&+(1/T)\lambda(\{t\in[T,2T]:|\operatorname{Im}\Sigma_{g,T^{1/V}}^*(t)
-\operatorname{Im}\Sigma_{g,T^{1/V}}(t)|\geq CV/4\})\nonumber\\
&+(1/T)\lambda(\{t\in[T,2T]:|\operatorname{Im}\Sigma_{g,T^{1/V}}(t)
-\operatorname{Im}\Sigma_{g,x}(t)|\geq CV/4\})\nonumber\\
&+(1/T)\lambda(\{t\in[T,2T]:|\operatorname{Im}\Sigma_{g-f,x}(t)|\geq CV/4\})\label{eq:fourte}.
\end{align}
If we choose Selberg's function $g(u)=e^{-2u}\min(1,2(1-u))$, we can apply \cite[Theorem 1]{S}, which says that, assuming RH, there exists constants $C,C'>0$ such that for $2\leq y\leq t^2$ and $t\geq 2$,
\begin{equation}\label{eq:selb}
|\operatorname{Im}r_{g,y}^*(t)|\leq\bigg|\frac{C'}{\log y}\sum_{n\leq y}\frac{\Lambda(n)}{n^{1/2+it}}g\bigg(\frac{\log n}{\log y}\bigg)\bigg|+\frac{C}{16}\frac{\log t}{\log y}.
\end{equation}
If we choose $y=T^{1/V}$ and $t\in[T,2T]$, $T\geq 2$, we have $(C/16)(\log t/\log y)\leq CV/8$. For $T\geq 2$, sufficiently large such that $2\leq T^{1/V}\leq T^2$, we obtain
\begin{multline*}
(1/T)\lambda(\{t\in[T,2T]:|\operatorname{Im}r_{g,T^{1/V}}^*(t)|\geq CV/4\})\\
 \leq\frac{1}{T}\lambda\bigg(\bigg\{t\in\left[T,2T\right]:\bigg|\frac{C'}{\log T^{1/V}}\sum_{n\leq T^{1/V}}\frac{\Lambda(n)}{n^{1/2+it}}g\bigg(\frac{\log n}{\log T^{1/V}}\bigg)\bigg|
  \geq CV/8 \bigg\}\bigg).
\end{multline*}
Now, we can apply Markov's inequality and~\eqref{eq:appmv} to bound the last term by
\begin{equation*}
\bigg(\frac{8C'}{CV}\bigg)^{2\lfloor V\rfloor}3^{2V}(2(AV)^V+O(1)^V)=e^{-(1-o(1))V\log V}.
\end{equation*}
Similarly, by using the other bounds in Appendix~\ref{AB}, we can bound the three other terms in~\eqref{eq:fourte} by $\exp(-(1-o(1))V\log V)$. Hence,~\eqref{eq:exb} follows. This completes the proof.
\end{proof}

\begin{proof}[Proof of Corollary~\ref{cor3}] The asserted formula is exactly content of Varadhan's integral lemma (see Theorem~\ref{tc2}). The assumptions of the theorem are satisfied by Corollary~\ref{cor2} and equation~\eqref{eq:c1}.
\end{proof}

\begin{appendix}
\section{Selberg's result}
In this appendix we briefly discuss Selberg's result about the rate of convergence in the central limit theorem of $\operatorname{Im}\log\zeta(1/2+it)$ (see \cite[Theorem 2]{S3} and \cite[Theorem 6.2]{T}). From Theorem~\ref{thma} we deduce:
  
\begin{lemma} Let $x=e^{\log T/N}$ and $N$ such that $x\rightarrow\infty$ and $N/\log\log T\rightarrow\infty$ as $T\rightarrow\infty$. Suppose further that $N/\log\log T=O(\log\log T)$. Then
\begin{multline}\label{eq:cltbe}\sup_{a<b}\Bigg(\ \frac{1}{T}\lambda\Big(\Big\{t\in[T,2T]:\frac{1}{\sqrt{(\log\log x +\gamma)/2}}\sum_{p\leq x}\frac{\sin(t\log p)}{\sqrt{p}}\in [a,b]\Big\}\Big)\\
-\int_a^be^{-t^2/2}\frac{dt}{\sqrt{2\pi}}\ \Bigg)=O(1/\sqrt{\log\log T}).
\end{multline}
\end{lemma}

\begin{proof} We denote by $\Phi_n(u)$ the left hand side of~\eqref{eq:i3}. Using \cite[XVI.3, formula 3.13]{F} we can bound the left hand side of~\eqref{eq:cltbe} by
\begin{equation}\label{eq:ab}
\frac{2}{\pi}\int_{-c\sqrt{\log\log x}}^{c\sqrt{\log\log x}}\limits e^{-u^2/2}|(\Phi_n(u/\sqrt{(\log\log x +\gamma)/2})-1)/u|du+O\Big(\frac{1}{c\sqrt{\log\log x}}\Big).
\end{equation}
An inspection of the proof of Proposition~\ref{prop2} combined with~\eqref{eq:p1}  shows that $\Phi_n(u)=\Phi(u)(1+O(1/\log x))+O(1/\log T)$, $|u|\leq c$. If we choose $c>0$ such that $\Phi(u)$ has no zeros for $|u|\leq c$, we obtain $\Phi_n(u)=\Phi(u)(1+O(1/\log x))$, $|u|\leq c$. On the other hand, we have $\Phi(u/\sqrt{(\log\log x +\gamma)/2})=1+O(u^2/\log\log x)$, $|u|\leq c\sqrt{\log\log x}$. Plugging in these estimates gives that~\eqref{eq:ab} is $O(1/\sqrt{\log\log x})$. From $N/\log\log T=O(\log\log T)$ we conclude that $\log\log T/\log\log x\rightarrow 1$ and this completes the proof.
\end{proof}

This lemma combined with the bound (see \cite[Lemma 6.2]{T})
\begin{equation*}
 |\{t\in[T,2T]:|r_{1,x}(t)|\geq c'\log\log\log T\}|=O(1/\sqrt{\log\log T}),
\end{equation*}
where $c'>0$ is a constant, yields Selberg's result
\begin{multline*}
\sup_{a<b}\Bigg(\ \frac{1}{T}\lambda\Big(\Big\{t\in[T,2T]: \frac{\operatorname{Im}\log\zeta(1/2+it)}{\sqrt{(\log\log T)/2}}\in [a,b]\Big\}\Big)\\
  -\int_a^be^{-t^2/2}\frac{dt}{\sqrt{2\pi}}\ \Bigg)=O\bigg(\frac{\log\log\log T}{\sqrt{\log\log T}}\bigg).
\end{multline*}

\section{Mean value estimates}\label{AB}

For completeness we present some standard mean value estimates which we applied in the proof of Corollary~\ref{cor2} (see \cite[Lemma 3]{S} and \cite[Lemma 3]{SO}). For this purpose let $x$ and $y$ be positive real numbers, $a_p$ and $b_p$ be complex numbers with $|a_p|\leq 1$ and $|b_p|\leq \log p/\log x$, and $k$ be a nonnegative integer. By repeating the arguments in the proof of Proposition~\ref{prop1}, we obtain
\begin{align*}
 &\frac{1}{T}\int_T^{2T}\bigg|\sum_{p\leq x}{\frac{a_p}{p^{1+2it}}}\bigg|^{2k}dt
 \leq k!\Big(\sum_{p\leq x}\frac{1}{p^2}\Big)^k+2D k!(\pi(x))^k/T,\\
 &\frac{1}{T}\int_T^{2T}\bigg|\sum_{y< p\leq x}{\frac{a_p}{p^{1/2+it}}}\bigg|^{2k}dt
  \leq k!\Big(\sum_{y< p\leq x}\frac{1}{p}\Big)^k+2D k!(\pi(x)-\pi(y))^k/T\\
 &\frac{1}{T}\int_T^{2T}\bigg|\sum_{p\leq x}{\frac{b_p}{p^{1/2+it}}}\bigg|^{2k}dt
  \leq k!\frac{1}{(\log x)^k}\Big(\sum_{p\leq x}\frac{\log p}{p}\Big)^k+2D k!(\pi(x))^k/T.
\end{align*}
If $x\leq T^{1/k}$, the first and the third term are bounded by $(Ak)^k$ and the second by 
$(k(\log\log x-\log\log y+A))^k$, $A>0$ some constant. 

For example, we obtain for a function $|g(u)|\leq 1$
\begin{align}
 \frac{1}{T}\int_T^{2T}\bigg|&\frac{1}{\log T^{1/V}}\sum_{n\leq T^{1/V}}\frac{\Lambda(n)}{n^{1/2+it}}g\Big(\frac{\log n}{\log T^{1/V}}\Big)\bigg|^{2\lfloor V\rfloor}dt\nonumber\\
 &=\frac{1}{T}\int_T^{2T}\bigg|\sum_{p\leq T^{1/V}}{\frac{b_p}{p^{1/2+it}}}+\sum_{p^2\leq T^{1/V}}{\frac{a_p}{p^{1+2it}}}+O(1)\bigg|^{2\lfloor V\rfloor}dt\nonumber\\
 &\leq 3^{2V}((AV)^V+(AV)^V+O(1)^V)\label{eq:appmv}.
 \end{align}
 
\section{Large deviation theory} \label{ldt}

 In this appendix we give the definition of the large deviation principle and state two important results which we used in the proofs of Corollary~\ref{cor2} and~\ref{cor3} (see \cite{DZ}).
 
A function $I:\mathbb{R}\rightarrow[0,\infty]$ is called a rate function (resp. good rate function), if for all $\alpha\in[0,\infty)$, the sets $\{x:I(x)\leq\alpha\}$ are closed (resp. compact). A family $\{Z_\epsilon\}$ of real-valued random variables satisfies the large deviation principle with the speed $\epsilon$ and the rate function $I$, if

(a) For any closed set $F\subseteq\mathbb{R}$
\begin{equation*}
 \limsup_{\epsilon\rightarrow 0}\epsilon\log\mathbb{P}(Z_\epsilon\in F)\leq-\inf_{x\in F}I(x).
\end{equation*}

(b) For any open set $G\subseteq\mathbb{R}$
\begin{equation*}
 \liminf_{\epsilon\rightarrow 0}\epsilon\log\mathbb{P}(Z_\epsilon\in G)\geq-\inf_{x\in G}I(x).
\end{equation*}
\begin{thm}[G\"{a}rtner-Ellis, see Theorem 2.3.6 or 4.5.20 in \cite{DZ}]\label{tc1} Suppose that for each $\lambda\in\mathbb{R}$
\begin{equation*}
 \Lambda(\lambda):=\lim_{\epsilon\rightarrow 0}\epsilon\log\mathbb{E}\big[ e^{\lambda Z_{\epsilon}/\epsilon}\big]
\end{equation*}
exists and that $\Lambda$ is differentiable. Then the family $\{Z_\epsilon\}$ satisfies the large deviation principle with the good rate function $I(x)=\sup_{\lambda\in\mathbb{R}}(\lambda x-\Lambda(\lambda))$.
\end{thm}
\begin{thm}[Varadhan, see Theorem 4.3.1 in \cite{DZ}]\label{tc2} Suppose that $\{Z_\epsilon\}$ satisfies the large deviation principle with a good rate function $I$ and let $h\in\mathbb{R}$. Assume further that for some $\gamma>1$
\begin{equation}\label{eq:varc}
 \limsup_{\epsilon\rightarrow 0}\epsilon\log\mathbb{E}\big[ e^{\gamma h Z_{\epsilon}/\epsilon}\big]<\infty.
\end{equation}
Then
\begin{equation*}
 \lim_{\epsilon\rightarrow 0}\epsilon\log\mathbb{E}\big[ e^{h Z_{\epsilon}/\epsilon}\big]=\sup_{x\in\mathbb{R}}(xh-I(x)).
\end{equation*}
\end{thm}
\begin{definition}[see Definition 4.2.10 in \cite{DZ}]\label{exequiv} Let $\{Z_\epsilon\}$ and $\{\tilde{Z}_\epsilon\}$ be two families of real-valued random variables, defined on the same probability space. Then $\{Z_\epsilon\}$ and $\{\tilde{Z}_\epsilon\}$ are called exponentially equivalent if for each $\delta>0,$ 
\begin{equation}
\limsup_{\epsilon\rightarrow 0}\epsilon\log\mathbb{P}(|Z_\epsilon-\tilde{Z}_\epsilon|>\delta)=-\infty.
\end{equation}
\end{definition}
\end{appendix}

\section*{Acknowledgements}
Finally, I sincerely would like to thank Prof. Ashkan Nikeghbali and Prof. Emmanuel Kowalski for their support and guidance during the preparation of this paper. Thanks also to the referee for comments leading to improvements in the presentation of the manuscript.

\bibliographystyle{abbrv}
\bibliography{lit}
\end{document}